\documentclass[11pt]{amsart}
\usepackage{amsfonts,amssymb,amscd,amsmath,enumerate,verbatim,calc,graphicx}

\newcommand{\owari}{\qedsymbol}
\def\NZQ{\mathbb}               

\def\ZZ{{\NZQ Z}}
\def\RR{{\NZQ R}}

\def\PP{{\NZQ P}}

\def\frk{\mathfrak}               

\def\Phi{{\frk N}}
\def\eb{{\bold e}}

\def\opn#1#2{\def#1{\operatorname{#2}}} 
\opn\chara{char} \opn\length{\ell} \opn\pd{pd} \opn\rk{rk}
\opn\projdim{proj\,dim} \opn\injdim{inj\,dim} \opn\rank{rank}
\opn\depth{depth} \opn\grade{grade} \opn\height{height}
\opn\embdim{emb\,dim} \opn\codim{codim}

\opn\Tr{Tr} \opn\bigrank{big\,rank}
\opn\superheight{superheight}\opn\lcm{lcm}
\opn\trdeg{tr\,deg}
\opn\reg{reg} \opn\lreg{lreg} \opn\ini{in} \opn\lpd{lpd}
\opn\size{size}\opn{\mult}{mult}
\opn\aff{aff} \opn\con{conv} \opn\relint{relint} \opn\cok{coker}
\opn\img{Im} \opn\cn{cn} \opn\inte{int} \opn\vol{vol}
\opn\link{link} \opn\star{star}
\opn\gr{gr}

\def\Oc{{\mathcal O}}
\def\Qc{{\mathcal Q}}

\def\Pc{{\mathcal P}}

\def\Sc{{\mathcal S}}
\def\Tc{{\mathcal T}}
\newtheorem{Theorem}{Theorem}[section]
\newtheorem{Lemma}[Theorem]{Lemma}
\newtheorem{Corollary}[Theorem]{Corollary}
\newtheorem{Proposition}[Theorem]{Proposition}
\newtheorem{Remark}[Theorem]{Remark}

\newtheorem{Example}[Theorem]{Example}

\newtheorem{Definition}[Theorem]{Definition}
\newtheorem{Problem}[Theorem]{Problem}

\numberwithin{equation}{section}
\begin{document}

\title{Almost Gorenstein homogeneous rings and their $h$-vectors}

\author{Akihiro Higashitani}
\address{Akihiro Higashitani,
Department of Mathematics, Kyoto Sangyo University, 
Motoyama, Kamigamo, Kita-Ku, Kyoto, Japan, 603-8555}
\email{ahigashi@cc.kyoto-su.ac.jp}

\begin{abstract}
In this paper, for the development of the study of almost Gorenstein graded rings, 
we discuss some relations between almost Gorensteinness of Cohen--Macaulay homogeneous rings and their $h$-vectors. 
Concretely, for a Cohen--Macaulay homogeneous ring $R$, we give a sufficient condition for $R$ 
to be almost Gorenstein in terms of the $h$-vector of $R$ (Theorem \ref{suff}) and we also characterize 
almost Gorenstein homogeneous domains with small socle degrees in terms of the $h$-vector of $R$ (Theorem \ref{hvector}). 
Moreover, we also provide the examples of almost Gorenstein homogeneous domains arising from lattice polytopes. 
\end{abstract}

\maketitle

\section{Introduction}

Recently, for the study of a new class of local or graded rings which are Cohen--Macaulay but not Gorenstein, 
{\em almost Gorenstein} local or graded rings were defined and have been studied. 
In this paper, for the further study of almost Gorenstein rings, 
we concentrate on almost Gorenstein homogeneous rings 
and investigate the $h$-vectors of almost Gorenstein homogeneous rings.

Originally, the notion of almost Gorenstein local rings of dimension one was introduced by Barucci and Fr{\"o}berg \cite{BF} 
in the case where the local rings are analytically unramified. 
After this work, Goto, Matsuoka and Phuong \cite{GMP} modified the definition of one-dimensional almost Gorenstein local rings, 
which also works well in the case where the rings are analytically ramified. 
As a further investigation of almost Gorenstein local rings, 
the definition of almost Gorenstein local or graded rings of higher dimension suggested by Goto, Takahashi and Taniguchi \cite{GTT}. 
We refer the reader to \cite{BF, GMP, GTT} for the detailed information on almost Gorenstein rings.

In addition, Matsuoka and Murai \cite{MaMu} have studied the almost Gorenstein Stanley--Reisner rings. 
They introduce the notion ``uniformly Cohen--Macaulay'' and they prove that a simplicial complex $\Delta$ is 
uniformly Cohen--Macaulay if and only if there exists an injection from the Stanley--Reisner ring $k[\Delta]$ 
to its canonical module $\omega_{k[\Delta]}$. This is a necessary condition for $k[\Delta]$ to be almost Gorenstein (with $a$-invariant 0). 
They also characterize the almost Gorenstein$^*$ simplicial complexes of dimension at most 2. See \cite[Section 3.3]{MaMu}.

Inspired by these works (especially \cite{GTT} and \cite{MaMu}), 
we will study higher dimensional almost Gorenstein {\em homogeneous} rings. 
The main goal of this paper is to give a characterization of almost Gorensteinness 
in terms of the Hilbert series, namely, the $h$-vector of the Cohen--Macaulay homogeneous ring.

For a Cohen--Macaulay homogeneous domain and the $h$-vector $(h_0,h_1,\ldots,h_s)$ of $R$, 
it is known by the work of Stanley \cite[Theorem 4.4]{StanleyHF} that 
$R$ is Gorenstein if and only if $h_i=h_{s-i}$ for $i=0,1,\ldots,\lfloor s/2 \rfloor$. (See, also \cite[Corollary 4.4.6]{BH}.) 
This says that the ``symmetry'' of the $h$-vector of $R$ characterizes the Gorensteinness of $R$. 
Hence, it is natural to think of some ``almost symmetry'' of the $h$-vector of $R$ 
may characterize the almost Gorensteinness of $R$. Although to give a complete characterization might be impossible, 
we will prove that a certain almost symmetry of the $h$-vector of a Cohen--Macaulay homogeneous ring 
can be a sufficient condition to be almost Gorenstein (Theorem \ref{suff}). 
In addition, in the case where the ring is a domain and the socle degree is small, 
we will characterize the almost Gorensteinness in terms of its $h$-vector (Theorem \ref{hvector}).

The structure of this paper is as follows. 
In Section \ref{pre}, we recall the definition of almost Gorenstein graded rings 
and prepare some propositions for the discussions later. 
In Section \ref{juubun}, we prove that a certain almost symmetry of the $h$-vector of the Cohen--Macaulay homogeneous ring 
implies the almost Gorensteinness. More precisely, let $R$ be a Cohen--Macaulay homogeneous ring 
and $(h_0,h_1,\ldots,h_s)$ its $h$-vector. We prove that if $h_i=h_{s-i}$ for $i=0,1,\ldots,\lfloor s/2 \rfloor -1$, 
then $R$ is almost Gorenstein (Theorem \ref{suff}). 
In Section \ref{domain}, we characterize the almost Gorenstein homogeneous domain in terms of its $h$-vector. 
More precisely, for a Cohen--Macaulay homogeneous domain $R$, let $(h_0,h_1,\ldots,h_s)$ be the $h$-vector of $R$. 
When $s=2$, the following three conditions are equivalent: (i) $R$ is almost Gorenstein; (ii) $R$ is Gorenstein; (iii) $h_2=1$. 
Moreover, when $s=3$, $R$ is almost Gorenstein if and only if $h_3=1$. (See Theorem \ref{hvector}.) 
Finally, in Section \ref{rei}, we supply some examples of almost Gorenstein homogeneous domains arising from lattice polytopes.

\subsection*{Acknowledgements} 
The authors would like to be grateful to Shiro Goto, Naoyuki Matsuoka and Naoki Taniguchi 
for their helpful comments and instructive discussions. 
The author is partially supported by JSPS Grant-in-Aid for Young Scientists (B) $\sharp$26800015.

\bigskip

\section{Preliminaries}\label{pre}

Let $R$ be a Cohen--Macaulay homogeneous ring of dimension $d$ over an algebraically closed field $k$ with characteristic 0. 
Let $a=a(R)$ be the $a$-invariant of $R$ and let $\omega_R$ be a canonical module of $R$. 
Note that $a(R)=-\min\{ j : (\omega_R)_j \not= 0\}$. 

\subsection{Definitions}

We collect some notation used throughout this paper. 
\begin{itemize}
\item For a graded $R$-module $M$, 
\begin{itemize}
\item let $\mu(M)$ denote the number of elements in a minimal system of generators of $M$ as an $R$-module; 
\item let $e(M)$ denote the multiplicity of $M$; 
\item note that we have in general the inequality 
\begin{align}\label{ookii}
e(M) \geq \mu(M); 
\end{align}
\item let $M(-\ell)$ denote the $R$-module whose underlying $R$-module is the same as that of $M$ 
and whose grading is given by $(M(-\ell))_n=M_{n-\ell}$ for all $n \in \ZZ$; 
\item let $[[ M ]]$ denote the Hilbert series of $M$, i.e., 
$$[[M]]=\sum_{n \in \ZZ}(\dim_k M_n) t^n.$$ 
\end{itemize}
\item We denote the Cohen--Macaulay type of $R$ by $r(R)$. Note that $r(R)=\mu(\omega_R)$. 
\item We say that $(h_0,h_1,\ldots,h_s)$ is the {\em $h$-vector} of $R$ if 
$$[[R]]=\sum_{n \geq 0}(\dim_k R_n) t^n=\frac{h_0+h_1t+\cdots+h_st^s}{(1-t)^d}$$ 
with $h_s \not=0$. It is well known that each $h_i$ is a nonnegative integer (since $R$ is Cohen--Macaulay) and $h_0=1$. 
When $(h_0,h_1,\ldots,h_s)$ is the $h$-vector of $R$, the index $s$ is called the {\em socle degree} of $R$. 
\end{itemize}

Let us recall the definition of the almost Gorenstein {\em graded} ring. 

\begin{Definition}[{\cite[Definition 1.5]{GTT}}]{\em 
We say that a Cohen--Macaulay graded ring $R$ is {\em almost Gorenstein} if there exists an exact sequence 
\begin{align}\label{ex_seq}
0 \rightarrow R \xrightarrow{\phi} \omega_R(-a) \rightarrow C \rightarrow 0
\end{align}
of graded $R$-modules with $\mu(C)=e(C)$, where $\phi$ is an injection of degree 0. 
}\end{Definition}

\subsection{Properties on $C$}

For a while, we consider the condition: 
\begin{align}\label{condition}
\text{there exists an injection $\phi : R \rightarrow \omega_R(-a)$ of degree 0}. 
\end{align}
This is a necessary condition for $R$ to be almost Gorenstein. 
Let $C=\cok(\phi)$. Then $C$ is a Cohen--Macaulay $R$-module of dimension $d-1$ 
if $C\not=0$ (see \cite[Lemma 3.1]{GTT}).

First, we see that the condition \eqref{condition} is satisfied in some cases. For example: 
\begin{Proposition}\label{domeinnobaai}
When $R$ is a domain, $R$ always satisfies the condition \eqref{condition}. 
\end{Proposition}
\begin{proof}
When $R$ is a domain, a canonical module $\omega_R$ can be taken as an ideal $I_R \subset R$ (\cite[Section 3.3]{BH}). 
Replace $\omega_R$ by a canonical ideal $I_R$ of $R$. 

Consider any $R$-homomorphism $\phi : R \rightarrow I_R(-a)$ of degree 0 with $\phi \not=0$. 
Take $x \in \ker(\phi)$. Then $\phi(x)=x\cdot\phi(1)=0$. Since $x \in R$ and $\phi(1) \in I_R \subset R$ and $R$ is a domain, 
we see $x=0$ or $\phi(1)=0$. Since $\phi(1) \not=0$, we obtain $x=0$, i.e., $\ker(\phi)=0$, as desired. 
\end{proof}

Next, we would like to describe $\mu(C)$ and $e(C)$ in terms of some invariants on $R$. 
\begin{Proposition}\label{myu-}
Assume that $R$ satisfies \eqref{condition}. Then $\mu(C)=r(R)-1$. 
\end{Proposition}
\begin{proof}
Since $\phi$ is a degree 0 injection and all elements of $(\omega_R(-a))_0$ can be its minimal generators, 
$\phi(1)$ can be an element of minimal generators of $\omega_R(-a)$. Hence, the assertion follows from the exact sequence \eqref{ex_seq}. 
\end{proof}

\begin{Proposition}\label{i-}
Assume that $R$ satisfies \eqref{condition}. Let $(h_0,h_1,\ldots,h_s)$ be the $h$-vector of $R$. Then we have 
\begin{align}\label{ccc}
[[C]] = \frac{\sum_{j=0}^{s-1}((h_s+\cdots+h_{s-j})-(h_0+\cdots+h_j))t^j}{(1-t)^{d-1}}.
\end{align}
In particular, we have $e(C)=\sum_{j=0}^{s-1}((h_s+\cdots+h_{s-j})-(h_0+\cdots+h_j))$. 
\end{Proposition}
\begin{proof}
We know that $[[ R ]]=\sum_{j=0}^sh_jt^j/(1-t)^d$ and it is also known (e.g., see \cite[Corollary 4.4.6]{BH}) that 
\begin{align*}
[[ \omega_R(-a)]]=\frac{\sum_{j=0}^s h_{s-j}t^j}{(1-t)^d}. 
\end{align*}
From the exact sequence \eqref{ex_seq}, we obtain $[[C]]=[[\omega_R(-a)]]-[[R]]$. 
On the other hand, since $M$ is a Cohen--Macaulay module of dimension $d-1$, 
$[[C]]$ looks like $\sum_{j \geq 0}h_j't^j/(1-t)^{d-1}$, 
where each $h_j'$ is a nonnegative integer. Hence we obtain the equality 
$$\frac{\sum_{j=0}^s (h_{s-j}-h_j)t^j}{(1-t)^d}=\frac{\left(\sum_{j \geq 0}h_j't^j\right)(1-t)}{(1-t)^d}.$$ 
This implies that $h_j'=(h_s+\cdots+h_{s-j})-(h_0+\cdots+h_j)$ for $j=0,1,\ldots,s$, as required. 
\end{proof}

From this proposition, we obtain: 
\begin{Corollary}\label{prop1}
Assume that $R$ satisfies \eqref{condition}. Let $(h_0,h_1,\ldots,h_s)$ be the $h$-vector of $R$. Then we have the inequality 
$$h_s+\cdots+h_{s-j} \geq h_0+\cdots+h_j$$ for each $j=0,1,\ldots,\lfloor s/2 \rfloor$, 
and $R$ is Gorenstein if and only if all the equalities hold, namely, $h_j=h_{s-j}$ for each $j=0,1,\ldots,\lfloor s/2 \rfloor$. 
\end{Corollary}
\begin{proof}
The above inequality directly follows from \eqref{ccc} and the Cohen--Macaulayness of $C$. 
Moreover, since $R$ is Gorenstein if and only if $C=0$, the second assertion also follows. 
\end{proof}

\begin{Remark}{\em 
Corollary \ref{prop1} is just an analogy of Stanley's inequality for the $h$-vectors of 
semi-standard Cohen--Macaulay domains (\cite[Theorem 2.1]{StanleyCMD}). 
In the case of Stanley--Reisner ring of a uniformly Cohen--Macaulay simplicial complex, 
this inequality is also shown in \cite[Proposition 2.7]{MaMu}. 
}\end{Remark}

\begin{Corollary}\label{tokuchou}
The following four conditions are equivalent: 
\begin{itemize}
\item[(a)] there exists an injection $\phi : R \rightarrow \omega_R(-a)$ of degree 0 such that $C=\cok(\phi)$ satisfies $\mu(C)=e(C)$, 
namely, $R$ is almost Gorenstein; 
\item[(b)] every injection $\phi : R \rightarrow \omega_R(-a)$ of degree 0 satisfies $\mu(C)=e(C)$;  
\item[(c)] we have 
\begin{align*}
r(R)-1=\sum_{j=0}^{s-1}((h_s+\cdots+h_{s-j})-(h_0+\cdots+h_j)); 
\end{align*}
\item[(d)] we have 
\begin{align}\label{hutousiki}
\dim_k (\omega_R \otimes k)_{-a+j} = (h_s+\cdots+h_{s-j})-(h_0+\cdots+h_j) 
\end{align}
for each $j=1,\ldots,s-1$. 
\end{itemize}
In particular, $\phi$ does not matter for the almost Gorensteinness of $R$. 
\end{Corollary}
\begin{proof}
First, we observe that all the elements of $\omega_R(-a)$ of degree 0 can be an element of minimal system of generators of $\omega_R(-a)$. 
Thus, by Proposition \ref{myu-} and Proposition \ref{i-}, we obtain the equivalence of (a) and (b) and (c). 

Next, the implication (d) $\Rightarrow$ (c) easily follows from the facts 
$r(R)=\sum_{j=0}^{s-1} \dim_k (\omega_R \otimes k)_{-a+j}$ and $h_s=\dim_k (\omega_R \otimes k)_{-a}$. 

Finally, consider the implication (c) $\Rightarrow$ (d). From the exact sequence \eqref{ex_seq}, we obtain the exact sequence 
$R \otimes k \rightarrow \omega_R(-a) \otimes k \rightarrow C \otimes k \rightarrow 0.$ 
Since $(R \otimes k)_j = 0$ for $j \geq 1$, by taking $(-)_j$ with $j \geq 1$ for this sequence, we obtain the isomorphism 
$$(\omega_R \otimes k)_{-a+j} \cong (C \otimes k)_j$$ for $j \geq 1$. 
In general, for a graded $R$-module $M$ of dimension $e$, we see that $\dim_k (M \otimes k)_j \leq h_j'$ if $[[M]]=\sum_{j \in \ZZ}h_j't^j/(1-t)^e$. 
Hence, we obtain the inequality $$\dim_k(\omega_R \otimes k)_{-a+j} \leq (h_s+\cdots+h_{s-j})-(h_0+\cdots+h_j)$$ 
for each $j=1,\ldots,s-1$ by Proposition \ref{i-}. 
Therefore, if $\dim_k(\omega_R \otimes k)_{-a+j} < (h_s+\cdots+h_{s-j})-(h_0+\cdots+h_j)$ for some $j$, 
since $r(R)=\sum_{j=0}^{s-1} \dim_k (\omega_R \otimes k)_{-a+j}$ and $h_s=\dim_k (\omega_R \otimes k)_{-a}$, 
we conclude that $r(R)-1 < \sum_{j=0}^{s-1} ((h_s+\cdots+h_{s-j})-(h_0+\cdots+h_j))$, as desired. 
\end{proof}

\bigskip

\section{A sufficient condition to be almost Gorenstein in terms of $h$-vectors}\label{juubun}

Let $R$ be a Cohen--Macaulay homogeneous ring of dimension $d$ satisfying the condition \eqref{condition} 
and let $(h_0,h_1,\ldots,h_s)$ be the $h$-vector of $R$. One of the main results of this paper is the following: 
\begin{Theorem}\label{suff}
If $h_i=h_{s-i}$ for $i=0,1,\ldots, \lfloor s/2 \rfloor -1$, then $R$ is almost Gorenstein. 
\end{Theorem}

This theorem should be compared with the following Stanley's theorem: 
\begin{Theorem}[{\cite[Theorem 4.4]{StanleyHF}}]\label{symm}
Assume that $R$ is a domain. 
Then $R$ is Gorenstein if and only if $h_i=h_{s-i}$ for $i=0,1,\ldots, \lfloor s/2 \rfloor$. 
\end{Theorem}

This theorem says that under the assumption $R$ is a domain, 
the symmetry of the $h$-vector of $R$ can be a necessary and sufficient condition for $R$ to be Gorenstein. 
On the other hand, Theorem \ref{suff} says that a certain ``almost symmetry'' of the $h$-vector of $R$ 
can be a sufficient condition for $R$ to be almost Gorenstein. 

\begin{proof}[Proof of Theorem \ref{suff}]
When $s$ is even, $h_{\lfloor s/2 \rfloor}=h_{s-\lfloor s/2 \rfloor}$ is automatically satisfied. 
Hence, $R$ is Gorenstein by Theorem \ref{symm}, in particular, almost Gorenstein. 
Assume that $s$ is odd. 

Let $C=\cok \phi$, where $\phi$ is an injection in \eqref{condition}. We may show that $\mu(C)=e(C)$. 
By Proposition \ref{i-}, we have 
$$[[C]]=\frac{(h_{(s+1)/2}-h_{(s-1)/2})t^{(s-1)/2}}{(1-t)^{d-1}}.$$ 
Hence, it is easy to see that $C$ satisfies $\mu(C)=e(C)(=h_{(s+1)/2}-h_{(s-1)/2})$, as required. 
\end{proof}

\bigskip

\section{Almost Gorenstein homogeneous domains with small socle degrees}\label{domain}

Assume that $R$ satisfies the condition \eqref{condition}. 
Let $(h_0,h_1,\ldots,h_s)$ be the $h$-vector of $R$. 
In this section, we discuss the condition for $R$ to be almost Gorenstein in terms of $h$-vectors 
in the case where $s$ is small. First, we observe the case $s=1$.

\begin{Theorem}[{\cite[Theorem 10.4]{GTT}}]
In the case $s=1$, $R$ is always almost Gorenstein. Moreover, in this case, $R$ is Gorenstein if and only if $h_1=1$. 
\end{Theorem}
\begin{proof}
In the statement of \cite[Theorem 10.4]{GTT}, 
the condition $a=1-d$ is equivalent to $s=1$ and the condition ``$Q(R)$ is a Gorenstein ring'' 
is automatically satisfied since $R$ satisfies \eqref{condition} (\cite[Lemma 3.1]{GTT}). 
Thus, \cite[Theorem 10.4]{GTT} directly implies the assertion. 
Moreover, the condition for $R$ to be Gorenstein follows from Theorem \ref{symm}. 
\end{proof}

Next, let us discuss the case $s>1$. 

It is well known that the inequality $h_s \leq r(R)$ holds in general. 
We say that $R$ is {\em level} if $h_s=r(R)$. (See, e.g., \cite{HibiASL}). 
One can see from \cite[Lemma 10.2]{GTT} that 
if $R$ is level and $s>1$, then $R$ is almost Gorenstein if and only if $R$ is Gorenstein.

Before considering the almost Gorensteinness of $R$, 
we give an upper bound for the Cohen--Macaulay types of Cohen--Macaulay homogeneous {\em domains} in Theorem \ref{upper}. 
We will use this for the proof of Theorem \ref{hvector}. 
The most part of the discussions for Theorem \ref{upper} is derived from \cite[Section 3]{Yanagawa}. 

\begin{Theorem}\label{upper}
Let $R$ be a Cohen--Macaulay homogeneous domain. Then 
\begin{align}\label{ineq}
r(R) \leq \sum_{i=2}^s h_i - (s-2)h_1.
\end{align}
\end{Theorem}
\begin{proof}
Take a nonzero divisor $x$ of $R$. Then we have $r(R)=r(R/(x))$ 
and the $h$-vectors of $R$ and $R/(x)$ are equal. Hence, by the similar argument to \cite[Lemma 3.1]{Yanagawa}, 
it suffices to show that for any set of points in uniform position $X \subset \PP^r$ ($r \geq 2$) 
and the homogeneous coordinate ring $R$ of $X$, the inequality \eqref{ineq} holds.

Assume that $R$ is the coordinate ring of a set of points in uniform position $X \subset \PP^r$ ($r \geq 2$). 
Let $s$ be the socle degree of $R$. 
For $i=1,\ldots,s-1$, let $$\varphi_i : S_1 \otimes (\omega_R)_{-s+i} \longrightarrow (\omega_R)_{-s+i+1}$$ 
be the multiplication map, where $S=\text{Sym}_k R_1$ is the coordinate ring of $\PP^r$. 
Then this map $\varphi_i$ is 1-generic (\cite[Proposition 1.5]{Yanagawa}). 
By the theory of 1-generic map, we have 
$\dim_k \img \varphi_i \geq \dim_k S_1 + \dim_k (\omega_R)_{-s+i} -1$ for each $i=1,\ldots,s-1$ (see \cite{E}). Hence 
\begin{align*}
\dim_k\img\varphi_i &\geq \dim_k S_1 + \dim_k (\omega_R)_{-s+i} -1 \\
&=h_1+1+(h_s+h_{s-1}+\cdots+h_{s-i+1})-1\\
&=h_1+h_s+h_{s-1}+\cdots+h_{s-i+1}. 
\end{align*}
On the other hand, we have $\dim_k (\omega_R)_{-s+i+1}=h_s+h_{s-1}+\cdots+h_{s-i}$. 
Thus, we see that $\dim_k (\omega_R)_{-s+i+1}-\dim_k\img\varphi_i \leq h_{s-i}-h_1$. 
By $\dim_k (\omega_R \otimes k)_{-s+i+1} \leq \dim_k (\omega_R)_{-s+i+1}-\dim_k\img\varphi_i$, we obtain the inequality 
\begin{align}\label{kagi}
\dim_k (\omega_R \otimes k)_{-a+i} \leq h_{s-i}-h_1 
\end{align}
for $i=1,\ldots,s-1$. (Note that $-a=-s+1$ in this setting.) 
Since we can see that $(\omega_R \otimes k)_i = 0$ for each $i \geq 1$ and $i \leq -s$ (\cite[Lemma 3.9]{Yanagawa}), 
we have $\mu (\omega_R) \leq \dim_k (\omega_R)_{-a}+\sum_{i=1}^{s-1}(\dim_k (\omega_R)_{-a+i}-\dim_k\img\varphi_i )$. Therefore, 
\begin{align*}
r(A)=\mu(\omega_R) &\leq \dim_k (\omega_R)_{-a} + \sum_{i=1}^{s-1}(\dim_k (\omega_R)_{-a+i}-\dim_k\img\varphi_i ) \\
&\leq h_s + \sum_{i=1}^{s-1}(h_{s-i}-h_1) =\sum_{i=2}^s h_i -(s-2)h_1. 
\end{align*}
\end{proof}

As an immediate corollary of Theorem \ref{upper}, we see the following. 
Note that the following has been already obtained in \cite{Yanagawa}. 
\begin{Corollary}[{\cite[Corollary 3.11]{Yanagawa}}]\label{kei_yanagawa}
If the socle degree of $R$ is $2$, then $R$ is level. 
\end{Corollary}
\begin{proof}
In this case, the inequality \eqref{ineq} implies $r(R) \leq h_2$. On the other hand, as mentioned above, 
we always have $r(R) \geq h_2$. Hence, $r(R)=h_2$. 
\end{proof}

\bigskip

When $s=2$ or $s=3$ and $R$ is a domain, 
we can characterize the almost Gorensteinness of $R$ in terms of the $h$-vector as follows: 
\begin{Theorem}\label{hvector}
Assume that $R$ is a domain. 
\begin{itemize}
\item[(a)] When $s=2$, the following conditions are equivalent: 
\begin{itemize}
\item[(i)] $R$ is almost Gorenstein; 
\item[(ii)] $R$ is Gorenstein; 
\item[(iii)] $h_2=1$. 
\end{itemize}
\item[(b)] When $s=3$, $R$ is almost Gorenstein if and only if $h_3=1$. 
\end{itemize}
\end{Theorem}
\begin{proof}
(a) When $s=2$, $R$ is always level by Corollary \ref{kei_yanagawa}. Hence, 
the equivalence (i) $\Leftrightarrow$ (ii) follows. Moreover, (ii) $\Leftrightarrow$ (iii) also holds by Theorem \ref{symm}. 

\noindent
(b) Let $C=\cok \phi$ in \eqref{condition}. From Proposition \ref{myu-} and Proposition \ref{i-}, we have 
$\mu(C)=r(R)-1$ and $e(C)=3(h_3-h_0)+h_2-h_1$, respectively. By the inequality \eqref{ineq} given in Theorem \ref{upper}, we see that 
\begin{align*}
e(C)-\mu(C) &= 3(h_3-h_0)+h_2-h_1 - (r(R) - 1) \\
&\geq 3(h_3-1)+h_2-h_1 - (h_2+h_3 - h_1 -1) \\
&=2(h_3-1).
\end{align*}
Hence, if $h_3 > 1$, then $R$ is never almost Gorenstein by $e(C) - \mu(C) > 0$. 

On the other hand, if $h_3=1$, then $R$ is almost Gorenstein by Theorem \ref{suff}, as required. 
\end{proof}

\begin{Remark}{\em 
In the case $s \geq 4$, we do not know any characterization of almost Gorensteinness in terms of $h$-vectors. 
However, any characterization of almost Gorensteinness using $h$-vectors might be impossible. 
For example, whether $R$ is level or not cannot be determined by $h$-vectors. See \cite[Section 3]{HibiASL}. 
The examples described in \cite[Section 3]{HibiASL} say that the Cohen--Macaulay type cannot be determined 
in terms of $h$-vectors in general. Hence, by considering Corollary \ref{tokuchou}, 
it is natural to think of we cannot characterize almost Gorensteinness in terms of $h$-vectors, either. 
}\end{Remark}

We remain the following: 
\begin{Problem}
If exists, find examples of Cohen--Macaulay homogeneous domains $R$ and $R'$ such that 
\begin{itemize}
\item the $h$-vectors of $R$ and $R'$ are equal; 
\item $R$ is almost Gorenstein; 
\item $R'$ is not almost Gorenstein. 
\end{itemize}
\end{Problem}

Moreover, we are also interested in the question: 
if $R$ is an almost Gorenstein homogeneous domain with $s \geq 2$, then does $h_s=1$ holds? 
This is true when $s=2$ and $s=3$ by Theorem \ref{hvector}. 
On the other hand, when $R$ is not a domain, this is not true even in the case $s=2$. 
In fact, by taking the {\em ridge sum} of Gorenstein simplicial complexes, 
we can obtain an almost Gorenstein Stanley--Reisner ring with $h_s \geq 2$. 
For the detail, see \cite{MaMu}. 

Actually, this is true in general. 
\begin{Theorem}\label{new}
Let $R$ be an almost Gorenstein homogeneous domain and $(h_0,h_1,\ldots,h_s)$ its $h$-vector with $s \geq 2$. 
Then $h_s=1$. 
\end{Theorem}
\begin{proof}
From Corollary \ref{tokuchou} (d), the inequality \eqref{hutousiki} holds. 
In particular, we see that $$\dim_k(\omega_R \otimes k)_{-a+1}=(h_s+h_{s-1})-(h_0+h_1).$$ 
On the other hand, since $R$ is a domain, we also obtain $$\dim_k(\omega_R \otimes k)_{-a+1} \leq h_{s-1}-h_1$$
by \eqref{kagi} for $i=1$. Therefore, we conslude that $h_s \leq h_0=1$, i.e., $h_s=1$. 
\end{proof}

\bigskip

\section{Examples : Almost Gorenstein Ehrhart rings}\label{rei}

In this section, we provide some examples of almost Gorenstein homogeneous domains 
which are given by {\em Ehrhart rings} of lattice polytopes. 

Before giving examples, let us recall what the Ehrhart ring of a lattice polytope is. 
Let $\Pc \subset \RR^d$ be a lattice polytope, which is a convex polytope all of whose vertices belong to 
the standard lattice $\ZZ^d$, of dimension $d$. Let $k$ be a field. Then we define the $k$-algebra $k[P]$ as follows: 
\begin{align*}
k[\Pc]=k[ {\bf X}^\alpha Z^n : \alpha \in n\Pc \cap \ZZ^d, \; n \in \ZZ_{\geq 0}], 
\end{align*}
where for $\alpha=(\alpha_1,\ldots,\alpha_d) \in \ZZ^d$, ${\bf X}^\alpha Z^n=X_1^{\alpha_1} \cdots X_d^{\alpha_d} Z^n$ 
denotes a Laurent monomial in $k[X_1^\pm, \ldots,X_d^\pm, Z]$ and $n\Pc=\{nv : v \in \Pc\}$. 
This $k[\Pc]$ is a normal Cohen--Macaulay graded domain of dimension $d+1$, 
where a grading is defined by $\deg ({\bf X}^\alpha Z^n) =n$ for $\alpha \in n\Pc \cap \ZZ^d$. 
The ring $k[\Pc]$ is called the {\em Ehrhart ring} of $\Pc$. 
Note that the Ehrhart ring is not necessarily standard graded, but is semi-standard graded, i.e., 
$k[\Pc]$ is a finitely generated module over the subring of $k[P]$ generated by all the elements of degree 1. 
Hence, the Hilbert series of $k[\Pc]$ (called the {\em Ehrhart series} of $\Pc$) is of the form: 
$$[[\; k[\Pc]\; ]]=\frac{\sum_{i=0}^s h_i^* t^i}{(1-t)^{d+1}}, \;\;\; h_i^* \in \ZZ_{\geq 0}.$$ 
The sequence of the coefficiens $h^*(\Pc)=(h_0^*,h_1^*,\ldots,h_s^*)$ appearing in the numerator of the Ehrhart series 
is called the {\em $h^*$-vector} (or the $\delta$-vector) of $\Pc$. 
It is known that the socle degree $s$ of the Ehrhart ring of $\Pc$ can be written as follows: 
\begin{align}\label{socle}
s=d+1-\min\{ m \in \ZZ_{> 0} : m \inte(\Pc) \cap \ZZ^d \not= \emptyset\}, 
\end{align}
where $\inte(\Pc)$ denotes the interior of $\Pc$. Thus, in particular, $s \leq d$. 
For more details on Ehrhart rings or $h^*$-vectors, we refer \cite[Part 2]{HibiRedBook}.

First, we supply examples of Ehrhart rings whose almost Gorensteinness can be guaranteed by Theorem \ref{suff}. 
Let $\eb_1,\ldots,\eb_d \in \ZZ^d$ be the standard basis for $\RR^d$, where each $\eb_i$ is the $i$th unit vector, 
and let $\con(X)$ denote the convex hull of a set $X \subset \RR^d$. 
\begin{Example}{\em 
Let $\Pc_1=\con(\{\pm\eb_1,\pm\eb_2,\pm\eb_3, \eb_1+\eb_2+2\eb_3\}) \subset \RR^3$. 
Then $\Pc_1$ is a lattice polytope of dimension 3 and we see that $k[\Pc_1]$ is standard graded and $h^*(\Pc_1)=(1,4,7,1)$. 
Hence, by Theorem \ref{suff}, $k[\Pc_1]$ is almost Gorenstein, while this is not Gorenstein. 

Moreover, let $\Pc_2=\con(\{\pm\eb_1,\ldots,\pm\eb_5, \eb_1+\cdots+\eb_4+2\eb_5\}) \subset \RR^5$ 
be a lattice polytope of dimension 5. Then we see that $k[\Pc_2]$ is standard graded and $h^*(\Pc_2)=(1,6,16,26,6,1)$. 
Hence, by Theorem \ref{suff} again, $k[\Pc_2]$ is also almost Gorenstein but not Gorenstein. 

In general, let $d=2e+1$ with $e \geq 1$ and 
let $\Pc_e$ be the convex hull of $\{\pm \eb_1,\ldots,\pm \eb_d, \eb_1+\cdots+\eb_{d-1}+2\eb_d\} \subset \RR^d$. 
We compute $h^*(\Pc_e)$ for small $e$'s and the following show the results. 
\begin{align*}
&e=3: \;\; h^*(\Pc_3)=(1, 8, 29, 64, 99, 29, 8, 1), \\ 
&e=4: \;\; h^*(\Pc_4)=(1, 10, 46, 130, 256, 382, 130, 46, 10, 1), \\ 
&e=5: \;\; h^*(\Pc_5)=(1, 12, 67, 232, 562, 1024, 1486, 562, 232, 67, 12, 1), \\ 
&e=6: \;\; h^*(\Pc_6)=(1, 14, 92, 378, 1093, 2380, 4096, 5812, 2380, 1093, 378, 92, 14, 1). 
\end{align*}
The Ehrhart rings of these polytopes are standard graded and these are almost Gorenstein by Theorem \ref{suff}. 
We expect that $k[\Pc_e]$ is always standard graded and almost Gorenstein for every $e$, 
although it seems difficult to give a precise proof. 
}\end{Example}

Next, we provide other examples of almost Gorenstein Ehrhart rings. The examples which we will give 
are obtained from the lattice polytopes arising from finite partially ordered sets (posets, for short).

Let $P=\{x_1,\ldots,x_n\}$ be a poset with a partial order $\prec$ and let 
$$\Oc(P)=\{(a_1,\ldots,a_n) \in \RR^n : a_i \geq a_j \text{ if }x_i \preceq x_j \text{ in }P, \;\; 0 \leq a_i \leq 1 \text{ for }i=1,\ldots,n\}.$$
This convex polytope $\Oc(P)$ is called the {\em order polytope} of $P$. 
It is known that $\Oc(P)$ is a lattice polytope 
and the Ehrhart ring $k[\Oc(P)]$ is often called the {\em Hibi ring} of $P$. It is known that: 
\begin{itemize}
\item $\dim \Oc(P)=\sharp P$; 
\item $k[\Oc(P)]$ is standard graded and an algebra with straightening laws on $P$ (\cite{Hibilat}); 
\item the socle degree of $k[\Oc(P)]$ is equal to $\sharp P -\max\{ \sharp C : C \text{ is a chain in }P\}$, 
where a chain is a totally ordered subset of $P$. 
\end{itemize}
For more detailed information on order polytopes, please consult \cite{StanleyEC}, 
or on Hibi rings, please consult \cite{EHHM, Hibilat, HibiASL} and the references therein.

We describe the $h^*$-vector of $\Oc(P)$ in terms of $P$. 
Let $P=\{x_1,\ldots,x_n\}$ be a poset with a {\em natural} partial order $\prec$, i.e., if $x_i \prec x_j$, then $i<j$. 
We say that a map $\sigma : P \rightarrow [n]$ from $P$ to the $n$-elements chain 
{\em order-preserving} if $x \preceq y$ implies $\sigma(x) \leq \sigma(y)$. 
We can identify an order-preserving map $\sigma : P \rightarrow [n]$ with a permutation 
$\begin{pmatrix}
1              &\cdots &n \\
\sigma^{-1}(1) &\cdots &\sigma^{-1}(n)
\end{pmatrix} \in S_n$. The set of all the permutations obtained in this way is denoted by $\Qc(P)$. 
For example, if $P=\{x_1,x_2,x_3,x_4\}$ is a poset having a natural partial order $x_1 \prec x_3, x_2 \prec x_3$ and $x_2 \prec x_4$, 
then $\Qc(P)=\{1234, 2134, 1243, 2143, 2413\} \subset S_4$. Here, we denote $a_1a_2 \cdots a_n$ instead of the usual notation 
$\begin{pmatrix}
1   &2   &\cdots &n \\
a_1 &a_2 &\cdots &a_n
\end{pmatrix} \in S_n$. 
Moreover, for $\pi \in S_n$, let $d(\pi)=\sharp\{ i \in \{1,\ldots,n-1\}: \pi(i)>\pi(i+1)\}$ and 
let $D(\pi)=\{ i \in \{1,\ldots,n-1\}: \pi(i)>\pi(i+1)\}$. (Clearly, $d(\pi)=\sharp D(\pi)$.)

\begin{Proposition}[{cf. \cite[Theorem 4.5.14]{StanleyEC}}]\label{h_order_poly}
Let $P=\{x_1,\ldots,x_n\}$ be a poset with a natural partial order and 
let $h^*(\Oc(P))(=h(k[\Oc(P)]))=(h_0,h_1,\ldots,h_s)$. Then 
$$h_i=\sharp \{ \pi \in \Qc(P) : d(\pi)=i\}$$
for each $i=0,1,\ldots,s$. 
\end{Proposition}

Now, we will consider the following poset: 
For $m \geq 3$, let $P_m=\{x_1,x_2,\ldots,x_{2m-1},x_{2m}\}$ be the poset with the natural partial order 
$x_1 \prec x_3 \prec \cdots \prec x_{2m-1}$, $x_2 \prec x_4 \prec \cdots \prec x_{2m}$ and $x_1 \prec x_{2m}$. 
Figure \ref{poset} below is the Hasse diagram of $P_m$. 
\begin{figure}[htb!]
\centering
\includegraphics[scale=0.3]{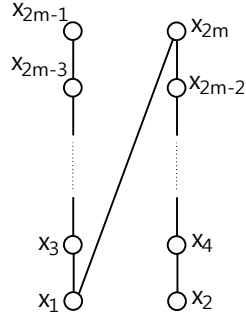}
\caption{the poset $P_m$}\label{poset}
\end{figure}

\begin{Theorem}\label{Hibiring}
Let $P_m$ be as above. 
Then $k[\Oc(P_m)]$ is an almost Gorenstein homogeneous domain whose $h$-vector $(h_0,h_1,\ldots,h_m)$ satisfies 
$h_i=h_{m-i}$ for $i=0,2,3,\ldots,\lfloor m/2 \rfloor$ and $h_{m-1}=h_1+1$. 
\end{Theorem}

Theorem \ref{Hibiring} says that there is an almost Gorenstein homogeneous ring having a ``different almost symmetry'' $h$-vector. 

\subsection{The $h$-vector of $R_m$} 
Let $R_m=k[\Oc(P_m)]$. Then we can compute by \eqref{socle} that 
the socle degree of $R_m$ is equal to $m$. 

First, for the proof of Theorem \ref{Hibiring}, let us observe the set $\Qc(P_m)$ in detail. Let 
\begin{align*}
\Sc=\{ \sigma \in S_{2m} : \; \sigma(1)<\sigma(3)<\cdots<\sigma(2m-1), 
\sigma(2)<\sigma(4)<\cdots<\sigma(2m), \sigma(1)<\sigma(2m)\}
\end{align*}
and let $\Sc^{-1}=\{\sigma^{-1} : \sigma \in \Sc\}$. Then $\Sc^{-1}=\Qc(P_m)$. 
Let $\tau=1325476\cdots (2m-4)(2m-1)(2m-2)2m \in S_{2m}$. Then we see that $\tau \in \Sc^{-1}$. 
Let $\Tc=\Sc^{-1} \setminus \{\tau\}$ and let $\Tc_i=\{\pi \in \Tc : d(\pi)=i\}$ for $i=0,1,\ldots,m$. 
We construct a map $\xi_i : \Tc_i \rightarrow \Tc_{m-i}$ for each $i=0,1,\ldots,m$ as follows:
Fix $\pi \in \Tc_i$. 
\begin{itemize}
\item Let $A=\{x_1,x_3,\ldots,x_{2m-1}\} \setminus \{x_{\pi(j)},x_{\pi(j+1)} : j \in D(\pi)\}$ and 
$B=\{x_2,x_4,\ldots,x_{2m}\} \setminus \{x_{\pi(j)},x_{\pi(j+1)} : j \in D(\pi)\}$. 
Since the parities of $\pi(j)$ and $\pi(j+1)$ are different for each $j \in D(\pi)$ and $j'+1 \not\in D(\pi)$ for each $j' \in D(\pi)$, 
we obtain that $\sharp A = \sharp B = m-i$. Let $p_1,p_2,\ldots,p_{m-i}$ (resp.  $q_1,q_2,\ldots,q_{m-i}$) be the positive integers 
with $A=\{x_{p_j} : 1 \leq j \leq m-i\}$ (resp. $B=\{x_{q_j} : 1 \leq j \leq m-i\}$) 
such that $p_1<p_2<\cdots<p_{m-i}$ (resp. $q_1<q_2<\cdots<q_{m-i}$). 
\item We construct the permutation $\widetilde{\pi}=a_1a_2 \cdots a_{2m}$ by defining each $a_i$ inductively as follows: 
Let $a_1=1$ if $p_1>1$ or let $a_1=2$ if $p_1=1$, and 
\begin{itemize}
\item for $1 \leq \ell <(p_1+q_1+1)/2$, let 
\begin{align*}
a_{\ell+1}=
\begin{cases}
a_\ell + 1, \; &\text{ if } a_\ell \leq \min\{p_1,q_1\}-2, \\
p_1, &\text{ if }a_\ell=q_1>p_1, \\
q_1, &\text{ if }a_\ell=p_1>q_1, \\
a_\ell+2, &\text{ otherwise}; 
\end{cases}
\end{align*}
\item for $(p_{j-1}+q_{j-1}+1)/2 \leq \ell < (p_j+q_j+1)/2$ with some $j \in \{2,\ldots,m-i\}$, let 
\begin{align*}
a_{\ell+1}=
\begin{cases}
a_\ell + 1, \; &\text{ if } \max\{p_{j-1},q_{j-1}\} + 1 \leq a_\ell \leq \min\{p_j,q_j\}-2, \\
p_{j-1}+2, &\text{ if }a_\ell \leq p_{j-1}, \; p_{j-1} > q_{j-1} \text{ and }a_\ell+2=q_j, \\
q_{j-1}+2, &\text{ if }a_\ell \leq q_{j-1}, \; q_{j-1} > p_{j-1} \text{ and }a_\ell+2=p_j, \\
p_j, &\text{ if }a_\ell = q_j>p_j, \\
q_j, &\text{ if }a_\ell = p_j>q_j, \\
a_\ell+2, &\text{ otherwise}; 
\end{cases}
\end{align*}
\item for $\ell \geq (p_{m-i}+q_{m-i}+1)/2$, let 
\begin{align*}
a_{\ell+1}=
\begin{cases}
a_\ell + 1, \; &\text{ if } a_\ell \geq \max\{p_{m-i},q_{m-i}\} + 1, \\
a_\ell+2, &\text{ otherwise}. 
\end{cases}
\end{align*}
\end{itemize}
See Example \ref{rensyuu} below for this construction. 
\item We define $\xi_i(\pi)=\widetilde{\pi}$. 
\end{itemize}

\begin{Example}\label{rensyuu}{\em 
Let $m=4$ and consider $\pi=13246857$. Then $d(\pi)=2$ and $D(\pi)=\{2,6\}$. 
Thus $A=\{x_1.x_7\}$, $B=\{x_4,x_6\}$ and $p_1=1, p_2=7, q_1=4, q_2=6$. 
We construct $\widetilde{\pi}=a_1a_2\cdots a_8$ as follows: 
\begin{align*}
&a_1=2 \text{ since }p_1=1, \; a_2=a_1+2=4, \; a_3=p_1=1, \\
&a_4=a_3+2=3, \; a_5=a_4+2=5, \; a_6=a_5+2=7, \; a_7=q_2=6 \text{ (since }a_6=p_2>q_2), \\
&a_8=a_7+2=8. 
\end{align*}
Actually, for $\pi \in \Tc_2$, we have $\widetilde{\pi}=24135768 \in \Tc_{4-2}$. 
}\end{Example}

\begin{Lemma}\label{hodai1} For this construction, 
\begin{itemize}
\item[(a)] $\widetilde{\pi} \in S_{2m}$ with $d(\widetilde{\pi})=m-i$; 
\item[(b)] the map $\overline{\pi}^{-1} : P \rightarrow [2m]$ 
defined by $\overline{\pi}^{-1}(x_j)=\widetilde{\pi}^{-1}(j)$ is order-preserving; 
\item[(c)] $\xi_i$ is bijective. 
\end{itemize}
\end{Lemma}
\begin{proof}
(a) For $(p_{j-1}+q_{j-1}+1)/2 \leq \ell < (p_j+q_j+1)/2$ with some $j \in \{2,\ldots,m-i\}$, by the definition of $a_\ell$, 
we can observe that $p_{j-1} \leq a_\ell \leq p_j$ when $a_\ell$ is odd or $q_{j-1} \leq a_\ell \leq q_j$ when $a_\ell$ is even. 
Moreover, we can also observe that for $(p_{j-1}+q_{j-1}+1)/2 \leq \ell, \ell' < (p_j+q_j+1)/2$ with $\ell \not= \ell'$, 
when the parities of $a_\ell$ and $a_{\ell'}$ are equal, $a_\ell < a_{\ell'}$ if and only if $\ell < \ell'$. 
Similar assertions also hold in the cases where $\ell < (p_1+q_1+1)/2$ and $\ell \geq (p_{m-i}+q_{m-i}+1)/2$. 
Thus, $a_1,a_2,\ldots,a_{2m}$ are all different integers with $1 \leq a_\ell \leq 2m$ for each $\ell$. 
Hence, $\widetilde{\pi}$ should define a permutation, i.e., $\widetilde{\pi} \in S_{2m}$. 
In addition, we also have $D(\widetilde{\pi})=\{(p_j+q_j-1)/2 : j=1,\ldots,m-i\}$. In particular, $d(\widetilde{\pi})=m-i$. 

\noindent
(b) As in (a) above, we know that when the parities of $a_\ell$ and $a_{\ell'}$ are equal, 
$a_\ell<a_{\ell'}$ if and only if $\ell<\ell'$. Moreover, one has $\widetilde{\pi}^{-1}(a_\ell)=\ell$ for each $\ell$. 

On the other hand, we know that if $a_\ell = 1$ and $a_{\ell'}=2m$, then $\ell < \ell'$. 
In fact, the situation $a_\ell=1$, $a_{\ell'}=2m$ and $\ell>\ell'$ may happen if $(p_1, q_1)=(1,2m)$ and $m-i=1$. 
However, this never happens because $\pi \not= 1325476\cdots (2m-4)(2m-1)(2m-2)2m = \tau$. 

Hence, the map $\overline{\pi}^{-1} : P \rightarrow [2m]$ defined by $\overline{\pi}^{-1}(x_{a_\ell})=\ell$ is order-preserving.

\noindent
(c) Since $\widetilde{\pi}$ is uniquely determined by $p_1,\ldots,p_{m-i},q_1,\ldots,q_{m-i}$, 
it is obvious that $\xi_i$ is injective. By the injectivities of $\xi_i : \Tc_i \rightarrow \Tc_{m-i}$ 
and $\xi_{m-i} : \Tc_{m-i} \rightarrow \Tc_i$ for each $i=0,1,\ldots,m$, we obtain the bijectivity of $\xi_i$ for each $i$, as required. 
\end{proof}

\subsection{The Cohen--Macaulay type of $R_m$}
Next, let us estimate the Cohen--Macaulay type of $R_m$ by considering its canonical ideal. 

Given a poset $P$, let $\hat{P} = P \cup \{\hat{0}, \hat{1}\}$ with $\hat{0} < x < \hat{1}$ for every $x \in P$. 
A map $v : \hat{P} \rightarrow \ZZ_{\geq 0}$ is called {\em order-reversing} (resp. {\em strictly order-reversing}) 
if $v(x) \leq v(y)$ (resp. $v(x)<v(y)$) for each $x,y \in \hat{P}$ with $x \succeq y$ (resp. $x \succ y$). 
Let $S(\hat{P})$ (resp. $T(\hat{P})$) denote the set of all order-reversing (all strictly order-reversing) maps 
$v : \hat{P} \rightarrow \ZZ_{\geq 0}$ with $v(\hat{1})=0$. 

It is shown in \cite{Hibilat} that $k[\Oc(P)]$ has a $k$-basis consisting of the monomials which look like 
$$Z^{v(\hat{0})} \prod_{p \in P}X_p^{v(p)}, \;\;\; v \in S(\hat{P}) $$
and it is also shown there that the monomials which look like 
$$Z^{v(\hat{0})} \prod_{p \in P}X_p^{v(p)}, \;\;\; v \in T(\hat{P}) $$
form a $k$-basis of the canonical ideal $I_{k[\Oc(P)]} \subset k[\Oc(P)]$. 
Let $T_0(\hat{P}) \subset T(\hat{P})$ denote the subset corresponding to the minimal set of generators of $I_{k[\Oc(P)]}$. 
Note that $k[\Oc(P)]$ is standard graded by the grading $\deg (Z^{v(\hat{0})} \prod_{p \in P}X_p^{v(p)}) = v(\hat{0})$. 

On the Cohen--Macaulay type of $R_m$, we see the following: 
\begin{Lemma}\label{hodai2}
We have $\sharp T_0(\widehat{P_m}) \geq m-1$, i.e., $r(R_m)=\mu(I_{R_m}) \geq m-1$. 
\end{Lemma}
\begin{proof}
For $i=1,\ldots,m-1$, let $v_i : P_m \rightarrow \ZZ_{\geq 0}$ be the map defined by 
\begin{align*}
v_i(x_j)=
\begin{cases}
m+1-\frac{j+1}{2}, \;\; &\text{ if $j$ is odd}, \\
m+i-\frac{j}{2}, \;\; &\text{ if $j$ is even} 
\end{cases}
\end{align*}
for $j=1,\ldots,2m$. Then each $v_i$ is a strictly order-reversing map. 

We will show that $v_i \in T_0(\widehat{P_m})$ for each $i=1,\ldots,m-1$. 
Assume that $v_i$ can be written as a sum of $v \in T(\widehat{P_m})$ and $w \in S(\widehat{P_m})$, 
i.e., $v_i(x)=v(x)+w(x)$ for each $x \in P$. Then we have $v_i(x_j)-v(x_j)=w(x_j) \geq 0$ for each $j$. 
Moreover, since $\hat{1} \succ x_{2m-1} \succ x_{2m-3} \succ \cdots \succ x_1 \succ \hat{0}$, 
$v \in T(\widehat{P_m})$ and $v(\hat{1})=0$, we have $v(x_{2q-1}) \geq m+1-q = v_i(x_{2q-1})$, 
equivalently, $v_i(x_{2q-1})- v(x_{2q-1}) \leq 0$ for each $q=1,2,\ldots,m$. 
Hence, we obtain $v_i(x_j)=v(x_j)$ for each $j=1,3,\ldots,2m-1$. 

On the other hand, if there is $1 \leq r < m$ with $v(x_{2r})-v(x_{2r+2}) \geq 2$, then we see that 
\begin{align*}
w(x_{2r+2})&=v_i(x_{2r+2})-v(x_{2r+2})=(v_i(x_{2r+2})+1)-(v(x_{2r+2})+1) \\
&\geq v_i(x_{2r})-(v(x_{2r}) - 1) > w(x_{2r}), 
\end{align*}
a contradiction to $w \in S(\widehat{P_m})$. Hence, $v(x_{2r})-v(x_{2r+2}) \leq 1$ for each $1 \leq r < m$, 
i.e., $v(x_{2r})-v(x_{2r+2}) = 1$ for each $r$ because $v \in T(\widehat{P_m})$. 
Let $v(x_{2m})=\alpha \in \ZZ_{>0}$. 
\begin{itemize}
\item Suppose that $\alpha>i$. Then we have $w(x_{2m})=v_i(x_{2m})-v(x_{2m})=i-\alpha < 0$, a contradiction. 
\item Suppose that $\alpha<i$. Then we have $w(x_{2m})> 0$. However, as observed above, 
we also have $w(x_1)=v_i(x_1)-v(x_1)=0$. Since $x_1 \prec x_{2m}$ in $P_m$, we should have $w(x_1) \geq w(x_{2m})$, a contradiction. 
\end{itemize}
Thus, $v(x_{2m})=i$. By $v(x_{2r})-v(x_{2r+2}) = 1$ for each $1 \leq r < m$, we also have $v(x_{2r})=m+i-r$. 
Therefore, $v=v_i$ and $w=0$. This implies that $v_i \in T_0(\widehat{P_m})$, as required. 
\end{proof}

\subsection{Proof of Theorem \ref{Hibiring}}
We are now in the position to give a proof of Theorem \ref{Hibiring}. 
Work with the same notation as above. 
Let $C=\cok \phi$, where $\phi : R_m \rightarrow I_{R_m}(-a)$ is an injection, 
and let $(h_0,h_1,\ldots,h_m)$ be the $h$-vector of $R_m$. 

By the above discussions and Lemma \ref{hodai1}, there exists a bijection $\xi_i$ between $\Tc_i$ and $\Tc_{m-i}$. 
In particular, we have $\sharp \Tc_i = \sharp \Tc_{m-i}$. 
On the other hand, by Proposition \ref{h_order_poly}, we have $h_i=\sharp \{\pi \in \Sc^{-1} : d(\pi)=i\}$ 
for each $i=0,1,\ldots,m$. Since we have $\Sc^{-1} = \bigcup_{i=0}^m \Tc_i \cup \{\tau\}$ and $d(\tau)=m-1$, 
we conclude that $(h_0,h_1,\ldots,h_m)$ satisfies that $h_i=h_{m-i}$ for $i=0,2,3,\ldots,\lfloor m/2 \rfloor$ 
and $h_{m-1}=h_1+1$. Therefore, by Proposition \ref{i-}, we obtain that 
$$e(C)=\sum_{j=0}^m ((h_m+\cdots+h_{m-j})-(h_0+\cdots+h_j))=\sum_{j=1}^{m-2}(h_{m-1}-h_1)=m-2.$$

In addition, it follows from Proposition \ref{myu-} and Lemma \ref{hodai2} that 
$\mu(C)=r(R_m)-1 \geq m-2.$ Hence, we see that $e(C) - \mu(C) \leq 0$. 
Since $e(C) - \mu(C) \geq 0$ is also satisfied by \eqref{ookii}, 
we conclude that $C$ satisfies $e(C)=\mu(C)$. Therefore, by Corollary \ref{tokuchou}, 
$R_m$ is almost Gorenstein, as desired. \hfill\owari

\begin{Example}{\em 
Let $Q_m=\{x_1,\ldots,x_{2m}\}$ be the poset with the partial order 
$x_1 \prec x_3 \prec \cdots \prec x_{2m-1}$, $x_2 \prec x_4 \prec \cdots \prec x_{2m}$, $x_1 \prec x_{2m}$ and $x_2 \prec x_{2m-1}$. 
By the similar discussions as above, we see that $k[\Oc(Q_m)]$ is also almost Gorenstein with $r(k[\Oc(Q_m)])=2m-3$. 
}\end{Example}

\end{document}